\documentclass[11pt,bezier]{article}
\usepackage{amsmath, amssymb, amsfonts, euscript, pifont, mathrsfs, latexsym, graphicx}

\newtheorem{prethm}{{\bf Theorem}}

\newenvironment{thm}{\begin{prethm}\sl{\hspace{-0.5
               em}{\bf.}}}{\end{prethm}}

\newtheorem{prepro}[prethm]{{\bf Proposition}}

\newtheorem{prelem}[prethm]{{\bf Lemma}}

\newenvironment{lem}{\begin{prelem}\sl{\hspace{-0.5
               em}{\bf.}}}{\end{prelem}}

\newtheorem{predeff}[prethm]{{\bf Definition}}

\newtheorem{precor}[prethm]{{\bf Corollary}}

\newenvironment{cor}{\begin{precor}\sl{\hspace{-0.5
               em}{\bf.}}}{\end{precor}}

\newtheorem{preconj}[prethm]{{\bf Conjecture}}

\newenvironment{conj}{\begin{preconj}\sl{\hspace{-0.5
               em}{\bf.}}}{\end{preconj}}

\newtheorem{preremark}[prethm]{{\bf Remark}}

\newtheorem{preexample}[prethm]{{\bf Example}}

\newtheorem{preproof}{{\bf\textsf{Proof.}}}

\newenvironment{proof}[1]{\begin{preproof}{\rm
               #1}\hfill{$\Box$}}{\end{preproof}}

\newcommand{\bmi}[1]{\mbox{\boldmath $ #1$}}
\DeclareMathAlphabet{\mathpzc}{OT1}{pzc}{m}{it}

\title{\bf\Large  On Order and Rank of Graphs}

\author{\large E. Ghorbani$^{\,\rm 1, 2}$ \quad \quad  A. Mohammadian$^{\,\rm 2}$ \quad \quad B. Tayfeh-Rezaie$^{\,\rm 2}$\\[.4cm]
{\sl $^{\rm 1}$Department of Mathematics, K.N. Toosi University of Technology,}\\
{\sl P.O. Box 16315-1618, Tehran, Iran}\\[0.3cm]
{\sl $^{\rm 2}$School of Mathematics, Institute for Research in Fundamental
Sciences (IPM),}\\{\sl P.O. Box
19395-5746, Tehran, Iran }
\\[0.5cm]{
$\mathsf{e\_ghorbani@ipm.ir}$ \quad\quad  $\mathsf{ali\_m@ipm.ir}$ \quad\quad  $\mathsf{tayfeh}$-$\mathsf{r@ipm.ir}$}}

\date{}

\begin{document}
\maketitle

\vspace{5mm}

\begin{abstract}
The rank of a graph is defined to be the rank of its adjacency matrix.
A graph is called  reduced if it has no isolated vertices and no two vertices  with the same set of neighbors.
Akbari, Cameron, and  Khosrovshahi conjectured that the number of vertices   of  every   reduced graph of rank $r$  is at most $m(r) = 2^{(r+2)/2}-2$ if $r$ is even and
$m(r) = 5 \cdot 2^{(r-3)/2} - 2$ if $r$ is odd. In this article, we prove  that if the conjecture is not true, then there would be   a  counterexample  of  rank  at most $46$.
We also show  that   every reduced graph of  rank $r$ has  at most $8m(r)+14$ vertices.

\vspace{5mm}
\noindent {\bf AMS Mathematics Subject Classification\,(2010):}   05C50,  15A03.
\end{abstract}

\vspace{5mm}

\section{Introduction}

For a  graph $G$, we denote  by $V(G)$ the vertex set of $G$. The {\sl order} of $G$ is the number of vertices of  $G$ and denoted by $|G|$.
The {\sl adjacency matrix}  of $G$, denoted by  $A(G)$,  has  its   rows and
columns  indexed by  $V(G)$  and its     $(u, v)$-entry  is $1$ if the vertices $u$ and
$v$ are adjacent and $0$ otherwise.
The {\sl rank} of  $G$, denoted by $\mathrm{rank}(G)$,  is the  rank of $A(G)$.

The problem of bounding the order of a graph in terms
of the   rank  was first studied by
Kotlov and   Lov\'asz  \cite{lov}.
Their motivation was to  determine  the gap between the   chromatic number and the  rank of graphs  originated
from the rank-coloring conjecture of van Nuffelen \cite{nuf}. The conjecture stated that
the  chromatic number of every
graph with at least one edge does not exceed the  rank. The first counterexample to the conjecture was obtained by Alon and
Seymour \cite{as}. A superlinear gap was  found by Razborov \cite{raz} and  a larger  gap
was provided by Nisan and Wigderson \cite{nw}. This problem, indeed, has a close connection with the    log-rank conjecture by Lov\'asz and Saks \cite{ls} from   communication complexity
which is equivalent to the statement  that the logarithm   of the chromatic number of any graph is bounded above by a polylogarithmic function of the rank, see \cite{ls1}.

The order of a graph with rank $r$ is  trivially  bounded above by   $2^r-1$ as soon as we make the assumption
that the graph is reduced; that is, it has no isolated vertices and no two vertices  with the same set of neighbors.
In fact,  over the two element field this bound is achievable by a unique graph \cite{god}.
We only consider
the rank of graphs  over the field of real numbers.
Kotlov and   Lov\'asz  \cite{lov}  proved that  there exists  a  constant $c$  such  that the order of every  reduced  graph of  rank $r$ is at most $c\cdot2^{r/2}$.
Later on,  Akbari, Cameron, and  Khosrovshahi  \cite{akb}    made the following conjecture.

\begin{conj}\label{ackconj}
For every integer $r\geqslant2$, the order of any     reduced  graph of  rank $r$
does not exceed   $m(r)$, where
\begin{eqnarray*}
m(r)=\left\{\begin{array}{ll}
2^{(r+2)/2}-2 &  \text{if $r$ is even}, \\
5\cdot2^{(r-3)/2}-2 & \text{if $r$ is odd}.
\end{array}\right.
\end{eqnarray*}
\end{conj}
They also constructed some reduced  graphs of rank $r$ and order $m(r)$, for every integer $r\geqslant2$.
In this article, we show  that if  Conjecture  \ref{ackconj} is not true, then there would be   a  counterexample  of rank  at most $46$. From our arguments,  it also follows that the order of every  reduced graph of  rank $r$  is at most $8m(r)+14$.

Recently, some relevant results were  obtained by a number of authors. Haemers and  Peeters  \cite{ham}  proved Conjecture \ref{ackconj} for  graphs  containing  an induced matching of size $r/2$  or an induced subgraph consisting a  matching of size $(r-3)/2$ and a cycle of length $3$.
Royle   \cite{roy} proved that the rank of every reduced graph containing no path of length $3$ as an
induced subgraph is equal to the order.
In \cite{gmt,tri},  we proved   that the order of every reduced tree, bipartite graph, and non-bipartite triangle-free graph of rank $r$  is at most $3r/2-1$,  $2^{r/2}+r/2-1$, and $3\cdot2^{\lfloor r/2\rfloor-2}+\lfloor r/2\rfloor$, respectively, and we characterized  all the  corresponding  graphs achieving these bounds.

\section{Notation and Preliminaries}\label{prel}

For a vertex $v$ of a graph  $G$, let $N_G(v)$ denote  the set of
all vertices of $G$ adjacent to $v$.
By $\mathnormal{\Delta}_G(u,v)$ we mean    the  symmetric difference of $N_G(u)$ and $N_G(v)$.
We will drop
the subscript $G$  when it is clear from the context.
Two vertices $u$ and $v$ of $G$ are called {\sl duplicated vertices} if $N(u)=N(v)$.
We say that  $G$ is {\sl reduced} if it has no isolated vertex  and no duplicated vertices.
A subset $S$ of  $V(G)$ with $|S|>1$  is called  a  {\sl duplication  class} of $G$  if  $N(u)=N(v)$
for any $u, v\in  S$.
For  a subset  $X$  of   $V(G)$,  $\langle X\rangle$ and  $G-X$ represent
the induced subgraphs of $G$ on $X$ and on $ V(G)\setminus X$, respectively. We use the same notation if $X$ is a subgraph of $G$. For a vertex $v\in V(G)$, we write  $G-v$ for  $G-\{v\}$.
For  a matrix $M$, we denote by $\mathrm{row}(M)$ the vector  space generated by the row vectors  of $M$ over the field of real numbers.
We  use  the notation  $\bmi{j}_k$  and  $J_{r\times s}$  for   the all one vector of length $k$ and  the  $r\times s$ all one matrix, respectively. The complete graph of order  $n$ is denoted by $\text{\sl{K}}_n$. For a graph $G$ with at least one edge, let $\rho(G)$ denote the minimum number of vertices whose removal results in a graph with a smaller rank.
If $G$ is not a  complete graph, then we  denote by $\tau(G)$ the minimum number of vertices whose removal results in a graph with  duplicated vertices.

\begin{lem}\label{lov0} {\rm\cite{kot, lov}}
For  any  reduced graph $G$,    the following hold.
\begin{itemize}
\item [{\rm (i)}] For every $v\in V(G)$, $\mathrm{rank}(G-N(v))\leqslant\mathrm{rank}(G)-2$.
\item [{\rm (ii)}] For every adjacent vertices $u,v\in V(G)$, $\mathrm{rank}(G-\mathnormal{\Delta}(u,v))\leqslant\mathrm{rank}(G)-1$.
\item [{\rm (iii)}]  For every non-adjacent vertices $u,v\in V(G)$, $\mathrm{rank}(G-\mathnormal{\Delta}(u,v))\leqslant\mathrm{rank}(G)-2$.
\item [{\rm (iv)}] If  $H$ is an induced  subgraph of $G$ with $|H|=|G|-\rho(G)$ and  $\mathrm{rank}(H)<\mathrm{rank}(G)$, then
$\mathrm{rank}(H)\geqslant\mathrm{rank}(G)-2$ and the equality occurs whenever  $H$ is not reduced.
\end{itemize}
\end{lem}

\begin{cor}\label{lova}
For  any  reduced graph $G$,
$$\rho(G)\leqslant\tau(G)=\min\big\{|\mathnormal{\Delta}(u,v)|\,\big|\,  u \textrm{ and }  v \textrm{ are distinct non-adjacent vertices of } G\big\}.$$
\end{cor}

The following lemma which has   a key role  in our proofs is inspired from \cite{lov}.

\begin{lem}\label{lov}
Let $G$ be a reduced graph  and let $H$ be  an induced  subgraph of $G$ with the  maximum  possible  order subject to that $H$ has duplicated vertices.  Let     $\mathrm{rank}(H)\geqslant\mathrm{rank}(G)-3$. Then the following  properties hold.
\begin{itemize}
\item [{\rm (i)}]   If $c$   is  an isolated vertex of $H$, then $N(c)=V(G-H)$.
\item [{\rm (ii)}] Every    duplication  class of  $H$  has    two elements and $H$ has at most  one isolated vertex.
\item [{\rm (iii)}]   One may label the duplication classes of $H$ as    $\{v_1, v_1'\}, \ldots, \{v_s, v_s'\}$ so that  there exist two  disjoint sets $T_1$ and  $T_2$  such that $V(G-H)=T_1\cup T_2$, $T_1\subseteq  N(v_i)\setminus  N(v_i')$ and  $T_2\subseteq  N(v_i')\setminus  N(v_i)$, for all $i\in\{1, \ldots, s\}$.
\item [{\rm (iv)}]   If both  $T_1$ and $T_2$  are non-empty, then $H$ has no isolated vertex.
\end{itemize}
\end{lem}

\begin{proof}{
For (i),  suppose that  $X= V(G)\setminus(V(H)\cup   N(c))$ is  non-empty.
Let $K=G-N(c)$.
If $u$ and  $v$ are duplicated vertices of $H$, then  by the definition of $H$, we find that $V(G-H)=\mathnormal{\Delta}_G(u,v)$ and so $X=\mathnormal{\Delta}_K(u,v)$. Therefore, Lemma \ref{lov0} implies that  $\mathrm{rank}(H)\leqslant\mathrm{rank}(K)-2\leqslant\mathrm{rank}(G)-4$, a contradiction.  For (ii), if $H$ has a  duplication  class containing three distinct vertices $x, y, z$,  then  for every vertex $t\in  V(G-H)$, at least one of
$\mathnormal{\Delta}(x, y)$, $\mathnormal{\Delta}(x, z)$, $\mathnormal{\Delta}(y, z)$
does not contain $t$. This contradicts the maximality of $H$.
The second part of (ii) follows from (i).
For (iii),   note that,   by the definition of $H$, every vertex in   $V(G-H)$ is  adjacent to exactly one vertex in each duplication class. If (iii) does not hold, then
$A(G)$ contains
$$\left[
\begin{array}{c|c}
\begin{array}{c|c|c|c|c}
    &  &  &  &  \\
    &  &  &  &  \\
  \bmi{x} & \bmi{x} & \bmi{y} &  \bmi{y}& \bmi{\star}  \\
    &  &  &  &  \\
    &  &  &  &
 \end{array}
&
\begin{array}{cc} 1&1\\0&0\\1&0\\0&1\\ \star & \star \vspace{-1mm}\end{array}
 \\\hline
 \begin{array}{ccccc} 1\,{} & 0\,{} & 1&{}\,0& {}\,\bmi{\star}\\1\,{} &  0\,{}  & 0&{}\,1& {}\,\bmi{\star}\end{array}
 &
 \begin{array}{cc}0& \star \\ \star&0 \end{array}
\end{array}
\right]$$
as a principle submatrix,
where  the     left-upper corner  is $A(H)$. This directly  yields that $\mathrm{rank}(H)\leqslant\mathrm{rank}(G)-4$, a contradiction.
For (iv), assume that   both  $T_1$ and $T_2$ are non-empty and $H$ has an isolated vertex.   Then, by   (i),
$A(G)$ contains
$$\left[
\begin{array}{c|c}
\begin{array}{c|c|c|c}
 &  &  &   \\
\bmi{x} & \bmi{x} & \bmi{\star} &\bmi{0} \\
 &  &   &  \\
 0&0 &\bmi{0}&0
\end{array}&
\begin{array}{cc} 1&0\\0&1\\ \star & \star\\1&1 \end{array}\\\hline
\begin{array}{cccc} 1{}\, & 0 & \bmi{\star} &{}\,1\\0{}\, &  1  & \bmi{\star}  &{}\,1
\end{array}&
\begin{array}{cc}
0&\star\\
\star&0 \end{array}
\end{array}
\right]$$ as a principle submatrix,
where  again the      left-upper corner   is $A(H)$. This directly  implies   that $\mathrm{rank}(H)\leqslant\mathrm{rank}(G)-4$, a contradiction.
}\end{proof}

Notice that for every integer
$r\geqslant4$, we have $m(r)=2m(r-2)+2$. Using this equality, we can prove the following lemma which will be  frequently used  in the  sequel.

\begin{lem}\label{rank}
Let $r$ and $k$ be two positive integers.
\begin{itemize}
\item [{\rm (i)}]  If  $r\geqslant6$ and $3\leqslant k\leqslant r-3$, then $m(k)+m(r-k)\leqslant m(r-2)+1$.
\item [{\rm (ii)}]  If  $r\geqslant10$ and $4\leqslant k\leqslant r-3$, then $m(k)+m(r-k+1)\leqslant m(r-2)$.
\end{itemize}
\end{lem}

\noindent{\bf{\textsf{Proof.}}}
For (i), we prove the statement by induction on $r$. For $r\in\{6,7, 8, 9\}$, (i) can be  easily verified. If   $k\in\{3, 4, r-4, r-3\}$,  then the inequality in (i) is clearly true.  For  $5\leqslant k\leqslant r-5$, by the induction hypothesis,  we have
\begin{align*}
m(k)+m(r-k)&=2m(k-2)+2m\big(r-4-(k-2)\big)+4\\&\leqslant2m(r-6)+6\\&=m(r-4)+4\\&<m(r-2)+1.
\end{align*}
For (ii), note that if  $k\in\{4,  r-3\}$,  then the inequality  is clearly valid.  If   $5\leqslant k\leqslant r-4$, then  using  (i),  we have
\begin{align*}
m(k)+m(r-k+1)&=2m(k-2)+2m\big(r-3-(k-2)\big)+4\\&\leqslant2m(r-5)+6\\&=m(r-3)+4\\&\leqslant m(r-2).
\end{align*}

${}$

\vspace{-17mm}

\hfill{$\Box$}

\vspace{3mm}

\section{Spherical codes}

In this section, we recall  some results on spherical codes. Let $n$ be a positive integer and $\varphi\in(0, \pi]$.
An {\sl $(n,M,\varphi)$-spherical code} $\mathscr{C}$ is a set of $M$ unit vectors in  $\mathbb{R}^n$  for which
$ \cos^{-1}(\bmi{\langle x}, \bmi{y\rangle})\geqslant\varphi$
 for every pair  $\bmi{x},\bmi{y}\in\mathscr{C}$, where $\bmi{\langle}\,\,, \,\,\bmi{\rangle}$  indicates the inner product of two vectors.
Let $\text{\sl{M}}(n, \varphi)$ denote the maximum possible value  $M$
for given $n$ and  $\varphi$ such that  an $(n,M,\varphi)$-spherical code exists.
We proceed to verify  the following lemma which is essential in the proof of our main theorem.

\begin{lem}\label{phius}
For every integer  $n\geqslant47$, $\text{\sl{M}}\big(n, \cos^{-1}(\sqrt{2}-1)\big)<5\cdot2^{(n-4)/2}-2$.
\end{lem}

The following theorem is due to Rankin \cite{Ran}.

\begin{thm}\label{rankin}
Let $n$ be a positive integer and $\varphi\in(0, \pi]$. Then $\text{\sl{M}}(n, \tfrac{\pi}{2})=2n$ and
\begin{eqnarray*}
\text{\sl{M}}(n, \varphi)\leqslant\left\{\begin{array}{lcc}
  n+1 & \,\,\,\text{ if } \varphi>\tfrac{\pi}{2}, \\ & \vspace{-2mm}\\
  \displaystyle{\frac{\sqrt{\pi}\mathnormal{\Gamma}\hspace{-1mm}\left(\tfrac{n-1}{2}\right)\sin\alpha\tan\alpha}{2\mathnormal{\Gamma}\hspace{-1mm}\left(\tfrac{n}{2}\right)\displaystyle{\int_{0}^{\alpha} (\sin\theta)^{n-2}(\cos\theta-\cos\alpha)\text{\sl{d}}\theta}}} & \text{ if } \varphi<\tfrac{\pi}{2},\\
\end{array}\right.
\end{eqnarray*}
where  $\alpha =\sin^{-1}\hspace{-1mm}\left(\sqrt{2}\sin\tfrac{\varphi}{2}\right)$ and $\mathnormal{\Gamma}$ denotes the Gamma function.
\end{thm}

From  \cite[p.\,97]{Zong}, we have
\begin{equation}\label{zon}
\int_{0}^{\alpha} (\sin\theta)^{n-2}(\cos\theta-\cos\alpha)\text{\sl{d}}\theta=\frac{(\sin\alpha)^{n+1}}{(n^2-1)\cos^2\alpha}\left(1-\frac{3\xi\tan^2\alpha}{n+3}\right),
\end{equation}
for some  $\xi\in[0,   1]$.
If    $n>\max\{6\tan^2\alpha-3, 5\}$, then $1-\tfrac{3\xi\tan^2\alpha}{n+3}>\tfrac{1}{2}$ and since $\mathnormal{\Gamma}$ is increasing on $[2, +\infty)$,   Theorem \ref{rankin} and (\ref{zon}) yield that
\begin{equation} \label{bdmn}
\text{\sl{M}}(n, \varphi)<\frac{n^2-1}{(\sin\alpha)^n}=(n^2-1)\left(\sqrt{2}\sin\tfrac{\varphi}{2}\right)^{-n}.
\end{equation}
Let $\varphi_{_0}=\cos^{-1}(\sqrt{2}-1)$. Then, by (\ref{bdmn}), we obtain that
$$\displaystyle{\text{\sl{M}}(n, \varphi_{_0})<(n^2-1)\left(1+\frac{1}{\sqrt2}\right)^{\tfrac{n}{2}}},$$
for every integer   $n>5$. So, it  is now easily  checked that $\text{\sl{M}}(n, \varphi_{_0})<5\cdot2^{(n-4)/2}-2$, for  every integer   $n\geqslant118$.

For smaller values of $n$,  we have  to employ   another upper bound for $\text{\sl{M}}(n, \varphi)$  given by  Leven\v{s}te\u{\i}n.
To present  the  Leven\v{s}te\u{\i}n bound, we first  recall that  the {\sl Gegenbauer polynomials} $Q_0(t), Q_1(t), \ldots$ which are defined by the recurrence relation
\begin{eqnarray*}
\left\{\begin{array}{l}
Q_0(t)=1;\\ Q_1(t)=t;\\
\displaystyle{Q_{k+1}(t)=\frac{(2k+n-2)tQ_{k}(t)-kQ_{k-1}(t)}{k+n-2}},  \text { for all }  k\geqslant1.
\end{array}\right.
\end{eqnarray*}
Now, let
$$Q_k^{1,0}(t)=\frac{(n-1)\big(Q_k(t)-Q_{k+1}(t)\big)}{(2k+n-1)(1-t)}\quad\text{and}\quad
Q_k^{1,1}(t)=\frac{(n-1)\big(Q_k(t)-Q_{k+2}(t)\big)}{(2k+n)(1-t^2)}.$$
For every integer  $k\geqslant1$, denote by  $t_k^{1,0}$ and $t_k^{1,1}$  the largest  zeros  of  $Q_k^{1,0}(t)$ and  $Q_k^{1,1}(t)$, respectively, and let $t_0^{1,1}=-1$.
We know from \cite[p.\,51]{Eric} that  $t_{k-1}^{1,1}<t_k^{1,0}<t_k^{1,1}$, for  every integer  $k\geqslant1$, and     $\{[t_{k-1}^{1,1}, t_k^{1,1})\,|\, k\geqslant1\}$ is a partition of $[-1, 1)$.
The following theorem is  called the Leven\v{s}te\u{\i}n bound \cite[p.\,57]{Eric}.

\begin{thm}\label{Levbd}
Let $n\geqslant3$  and $\varphi\in(0, \pi]$. Then
\begin{eqnarray*}
\text{\sl{M}}(n, \varphi)\leqslant\left\{\begin{array}{lcc}
\displaystyle{{k+n-3 \choose k-1}\left(\frac{2k+n-3}{n-1}-\frac{Q_{k-1}(s)-Q_{k}(s)}{(1-s)Q_k(s)}\right)} & \,\,\,\text{ if } s\in\left[t_{k-1}^{1,1}, t_k^{1,0}\right), \\ & \\
\displaystyle{{k+n-2 \choose k}\left(\frac{2k+n-1}{n-1}-\frac{(1+s)(Q_{k}(s)-Q_{k+1}(s)}{(1-s)(Q_k(s)+Q_{k+1}(s)}\right)} & \text{ if } s\in\left[t_k^{1,0}, t_k^{1,1}\right),\\
\end{array}\right.
\end{eqnarray*}
where  $s=\cos\varphi$.
\end{thm}

By  Theorem \ref{Levbd} and using   $\mathsf{Maple}$ for  computations,  we find that $\text{\sl{M}}(n, \varphi_{_0})<5\cdot2^{(n-4)/2}-2$, for every integer   $47\leqslant n\leqslant118$.
This discussion  completes the proof of Lemma \ref{phius}.

\section{Main Results}

In this section,  we present our main results.
We remark that   Conjecture \ref{ackconj} was verified for all graphs of  rank at most 8 by computation  \cite{akb}.
We have extended this  result to all graphs  of rank 9 by a computer search.

\begin{lem}\label{a1} Let $G$ be a reduced graph of order $n$ and rank $r\geqslant46$. If $n\geqslant5\cdot2^{(r-3)/2}-2$, then $\rho(G)<\big(1-\tfrac{1}{\sqrt{2}}\big)n$.
\end{lem}

\begin{proof}{
Suppose that $\rho(G)\geqslant\big(1-\tfrac{1}{\sqrt{2}}\big)n$. Let $M$ be the matrix  resulting    from   replacing all $0$ by $-1$ in $A(G)$. Clearly, $\mathrm{rank}(M)\leqslant r+1$ and by  Corollary \ref{lova},
$$\bmi{\langle}x, y\bmi{\rangle}\leqslant\frac{n-2\rho(G)}{n}\leqslant\sqrt{2}-1,$$
for every  pair  $x, y$ of  the row vectors  of  $\tfrac{1}{\sqrt{n}}M$.
It turns out that there are $n$  vectors in $\mathbb{R}^{r+1}$ where the angle  between each pair  of them  is at least  $\cos^{-1}\hspace{-1mm}\left(\sqrt{2}-1\right).$
In view of Lemma \ref{phius}, we have $n<5\cdot2^{(r-3)/2}-2$, a contradiction.}
\end{proof}

\begin{lem}\label{m2} Let $G$ be a reduced graph of order $n$ and rank $r\geqslant6$ . If $n>m(r)$, then $\rho(G)<n/2$.
\end{lem}

\begin{proof}{ If $\rho(G)\geqslant n/2$, then $\tfrac{n-2\rho(G)}{n}\leqslant0$. This, similar to the proof of Lemma \ref{a1}, implies  the existence of $n$  vectors in $\mathbb{R}^{r+1}$ such that
the angle  between each pair  of  which is at least $\frac{\pi}{2}$. From Theorem \ref{rankin},  it follows  that  $m(r)<n\leqslant2(r+1)$, which contradicts  $r\geqslant6$.}
\end{proof}

In what follows, we assume  that
$\mathbb{G}$ is  a counterexample  to Conjecture \ref{ackconj}  with the minimum possible order. Let
$n=|\mathbb{G}|$, $r=\mathrm{rank}(\mathbb{G})$, $\tau=\tau(\mathbb{G})$, and let
$H$ be an induced  subgraph of $\mathbb{G}$ of order $n-\tau$ with
duplicated vertices.
If $\mathrm{rank}(H)\geqslant r-3$, then by Lemma \ref{lov}\,(iii), we
may assume  that  $\{\upsilon_1, \upsilon_1'\}, \ldots, \{\upsilon_s, \upsilon_s'\}$ are  the duplication  classes of $H$.
For simplicity, let  $S=\langle\{\upsilon_1,\ldots,  \upsilon_s\}\rangle$ and $S'=\langle\{\upsilon_1',\ldots,  \upsilon_s'\}\rangle$.
Further, put   $T=\mathbb{G}- H$ and let  $T_1$ and $T_2$ be the sets  given  in Lemma \ref{lov}\,(iii)  with sizes $t_1$ and $t_2$, respectively.
We  denote the  number of isolated vertices of $H$  by $\epsilon$. Note that by Lemma \ref{lov}\,(ii), $\epsilon\in\{0, 1\}$.
Finally, we set $P=H-(V(S)\cup V(S')\cup\{c\})$ and $p=|P|$, where $c$ is the possible isolated vertex of $H$.

\begin{lem}\label{m1}
$n=m(r)+1$.
\end{lem}

\begin{proof}{
Let $v\in V(\mathbb{G})$.  If $\mathbb{G}-v$ is reduced, then by the minimality of $\mathbb{G}$, we have $|\mathbb{G}-v|\leqslant m(r)$ and so $n=m(r)+1$.
If $\mathbb{G}-v$ is not reduced, then either there is a vertex  $x\in V(\mathbb{G})$ such that $N(x)=\{v\}$ or there are two non-adjacent vertices $y, y'\in V(\mathbb{G})$ such that $\mathnormal{\Delta}(y, y')=\{v\}$. Hence, by Lemma \ref{lov0},  $\mathrm{rank}(\mathbb{G}-v)=r-2$. Therefore,  Lemma \ref{lov}\,(iii) yields that
every duplication class of $H$  has two vertices. Thus  $\tfrac{n}{2}-1\leqslant m(r-2)$. This is a contradiction as $m(r)=2m(r-2)+2$.}\end{proof}

\begin{lem}\label{t1g}
If $\tau\leqslant  m(r-2)+2$, then $\mathrm{rank}(H)\geqslant r-3$.
\end{lem}

\begin{proof}{
Suppose that $\tau\leqslant  m(r-2)+2$ and $\mathrm{rank}(H)\leqslant r-4$. Add a vertex from $ V(\mathbb{G}-H)$ to $H$ and call the resulting graph $K$. Obviously, $K$  has no duplicated vertices and $\mathrm{rank}(K)\leqslant r-2$. Thus
$n-\tau+1-\epsilon\leqslant m(r-2)$. This implies that   $n\leqslant m(r)$, a contradiction.}
\end{proof}

\begin{thm}\label{aslasl}
Suppose that  $\mathrm{rank}(H)\geqslant r-3$ with $r\geqslant10$. Then  $\epsilon=0$ and   one of the following holds.
\begin{itemize}
\item [{\rm (i)}] $S=\text{K}_1$ and  $\tau\geqslant m(r-2)+2$.
\item [{\rm (ii)}]  $S=\text{\sl{K}}_2$ and  $\tau\geqslant m(r-2)+1$.
\item [{\rm (iii)}] $S=\text{\sl{K}}_3$ and  $\tau=m(r-2)$.
\end{itemize}
\end{thm}

\begin{proof}{
We denote  the  possible  isolated vertex  of $H$ by $c$.
Also,  let $k=\mathrm{rank}(S)$,  $K=\langle V(T)\cup V(S)\rangle$ and  $K'=\langle V(T)\cup V(S')\rangle$.
We first establish  the following steps.

\noindent{\bf{\textsf{Step 1.}}}  $s+p\leqslant m(r-2)$, $\tau+s\geqslant m(r-2)+3-\epsilon$, and $\tau\geqslant p+3-\epsilon$.

Applying   Lemma \ref{lov0}\,(iii),   $\mathrm{rank}(H)\leqslant r-2$ and so $\mathrm{rank}(\langle V(S)\cup V(P)\rangle)\leqslant r-2$. 
By  the definitions of $S$ and $P$, $\langle V(S)\cup V(P)\rangle$   is a  reduced graph and thus  $s+p\leqslant m(r-2)$.  
Moreover,  $n=m(r)+1$ and $n=\tau+2s+p+\epsilon$ imply that  $\tau+s\geqslant m(r-2)+3-\epsilon$. By subtracting these inequalities, we obtain the last inequality.

\noindent{\bf{\textsf{Step 2.}}}  The graph $S$ has no  duplication classes.

By contradiction, suppose  that there are  two vertices $a, b\in S$ with $  N_{S}(a)=  N_{S}(b)$. Hence  $\mathnormal{\Delta}(a,b)\subseteq V(P)$  and by Corollary \ref{lova}, we obtain that $\tau\leqslant p$,  which is a contradiction to Step 1.

\noindent{\bf{\textsf{Step 3.}}}  If $S$ has isolated vertices, then  both  $T_{1}$ and $T_{2}$ are non-empty.

By contradiction, assume   that $\upsilon_1$ is an isolated vertex of $S$ and  $T_{1}$ is empty. Thus $N(\upsilon_1)\subseteq V(P)$.
We show that $\mathbb{G}-(N(\upsilon_1)\cup\{\upsilon_1\})$ is reduced.
If  $\mathbb{G}-(N(\upsilon_1)\cup\{\upsilon_1\})$ has   an isolated vertex, say $x$, then $x$ is not adjacent to $\upsilon_1$ and  $\mathnormal{\Delta}(x, \upsilon_1)\subseteq N(\upsilon_1)$, and if $\mathbb{G}-(N(\upsilon_1)\cup\{\upsilon_1\})$ has   a duplication class, say $\{y, y'\}$, then  $\mathnormal{\Delta}(y, y')\subseteq N(\upsilon_1)$. Since   $|N(\upsilon_1)\cup\{\upsilon_1\}|<p+3-\epsilon\leqslant\tau$, both cases contradict the  minimality of $\tau$ using Lemma \ref{lov0}.
So $\mathbb{G}-(N(\upsilon_1)\cup\{\upsilon_1\})$ is a reduced graph of order at least  $n-p-1$ and rank at most $r-2$. This implies that $p\geqslant m(r-2)+2$,  which is a contradiction to Step 1.

\noindent{\bf{\textsf{Step 4.}}}   Every duplication class of $T$ consists of  one vertex from $T_1$ and one from $T_2$.

Otherwise, without loss of generality, suppose that there are  two vertices $a, b\in T_1$ such that $N_T(a)=N_T(b)$. Therefore, $\mathnormal{\Delta}(a,b)\subseteq V(P)$ and so $\tau\leqslant p$,  which is a contradiction to Step 1.

\noindent{\bf{\textsf{Step 5.}}}  $\mathrm{rank}(K)\geqslant r-1$ and  $\mathrm{rank}(K')\geqslant r-1$.

We only prove  that $\mathrm{rank}(K)\geqslant r-1$.  By Step 1, $|K|=\tau+s\geqslant m(r-2)+3-\epsilon$. We show  that $K$ has a reduced  induced subgraph of  order  at least $m(r-2)+1$ which in turn
implies    that $\mathrm{rank}(K)\geqslant r-1$ by  the minimality of $\mathbb{G}$.
If  $K$ has no duplication classes, then  $K$ has at  most  one isolated vertex. Thus,  after removing the possible isolated vertex  from
$K$, we obtain   the desired subgraph.
So, assume that $K$ has  duplication classes.
By applying Steps 2, 3, and 4,  it is  easily checked that $T_1$ is non-empty  and  $K$ has exactly one duplication class which is of the form  $\{\upsilon_1, x\}$, for some $x\in T_2 $.
Hence   $K$ has at most one isolated vertex. Furthermore,   Lemma \ref{lov}\,(iv)
implies that $\epsilon=0$.
Now,  after removing the possible isolated vertex  from
$K-\upsilon_1$,  we obtain   the desired subgraph.

\noindent{\bf{\textsf{Step 6.}}}   The graph $T$  has no isolated vertices.

By contradiction, without loss of generality, assume  that $N_{T}(a)$ is empty,   for some   $a\in T_1$. Then $N(a)\subseteq  V(S)\cup V(P)\cup\{c\}$. Since  $K'=\mathbb{G}- (V(S)\cup  V(P)\cup\{c\})$, we deduce that  $\mathrm{rank}(K')\leqslant\mathrm{rank}(\mathbb{G}-  N(a))\leqslant r-2$, which is a  contradiction to Step 5.

\noindent{\bf{\textsf{Step 7.}}} Both   $T_1$ and $T_2$ are non-empty.

If $T_1$ is empty, then  $\mathrm{rank}(\langle V(T)\cup\{c\}\rangle)+\mathrm{rank}(S)\leqslant r$.  By  Steps 2, 3, 4,   and 6,  $\langle V(T)\cup\{c\}\rangle$ and $S$ are reduced graphs. So,   Step 1 implies that  $m(r-2)+3\leqslant \tau+s+\epsilon\leqslant m(r-k)+m(k)$,  which contradicts Lemma \ref{rank}\,(i), since $2\leqslant k\leqslant r-2$.
Similarly, we see that $T_2$ is non-empty.

\noindent{\bf{\textsf{Step 8.}}}   $\epsilon=0$.

It immediately follows from Step 7 and Lemma \ref{lov}\,(iv).

We now proceed with the following  cases.

\noindent{\bf{\textsf{Case 1.}}}  Assume that  $T$ has  a duplication class. We prove   that $S=\text{\sl{K}}_2$, $\mathrm{rank}(T)=r-3$, $\tau=m(r-2)+1$,  and $p=m(r-2)-2$. Since $T$ has a duplication class, $(\bmi{j}_{t_{1}}, \bmi{0})\not\in\mathrm{row}(A(T))$. By   Step 4, the two row vectors     of
$$X=\left[
\begin{array}{c|c}
A(T)   &\hspace{-1.7mm}\begin{array}{c} J_{t_{1}\times s}\\\hline O \end{array} \\
\end{array}\right]$$
 corresponding to  a  duplication class of $T$ are linearly independent. Extend these vectors  to  a  basis $\EuScript{B}$ of size $\mathrm{rank}(T)+1$ for $\mathrm{row}(X)$.
It is straightforward to see that  the row vectors  of
$$Y=\left[
\begin{array}{c|c}
A(T) &  \hspace{-2.7mm}\begin{array}{c|c} J_{t_{1}\times s}& O \\\hline O &  J_{t_{2}\times s} \end{array} \\\hline
\begin{array}{c|c} J_{s\times t_{1}} & O\\\hline  O & J_{s\times t_{2}} \end{array} \hspace{-1.7mm} & \hspace{-3.7mm}\begin{array}{c|c} \hspace{2mm}A(S) & \hspace{1mm} A(S)\\\hline \hspace{2mm}A(S) & \hspace{1mm} A(S)\end{array}
\end{array}
\right]$$
corresponding to   $\EuScript{B}$ along with  the  row vectors  of $Y$  corresponding to a basis for $\mathrm{row}(A(S))$  are  linearly independent. This implies that $\mathrm{rank}(T)+\mathrm{rank}(S)\leqslant r-1$. Note that  by   Step 4, the maximum reduced subgraph of $T$ has at least $\tau/2$ vertices. Moreover,   since $\mathrm{rank}(K)\geqslant r-1$, it is not hard to show  that $\bmi{j}_{s}\in\mathrm{row}(A(S))$ and so by   Step 2, $S$  is reduced.  Now, from   Steps 1, 6,  and 8, we have $m(r-2)+3\leqslant \tau+s\leqslant2m(r-k-1)+m(k)=m(r-k+1)+m(k)-2$. Applying  Lemma \ref{rank}\,(ii), we find that $k=2$ and hence  $S=\text{\sl{K}}_2$. Since $\tau\geqslant m(r-2)+1$, we deduce that  $\mathrm{rank}(T)=r-3$. If  $\{a, b\}$ is  a duplication class of $T$, then $\mathnormal{\Delta}(a,b)\subseteq  V(H)$ and therefore  by Corollary  \ref{lova},  $\tau\leqslant p+4$. On the other hand, by Step 1, we have  $\tau\geqslant p+3$ and since $n=\tau+p+4$, it follows that  $\tau=m(r-2)+1$ and $p=m(r-2)-2$, as required.

\noindent{\bf{\textsf{Case 2.}}}  Assume that $T$ has  no duplication classes.

\noindent{\bf{\textsf{Subcase 2.1.}}}   $(\bmi{j}_{t_{1}}, \bmi{0})\not\in\mathrm{row}(A(T))$ and  $\bmi{j}_{s}\not\in\mathrm{row}(A(S))$.

Since $\mathrm{rank}(X)=1+\mathrm{rank}(T)$ and $\bmi{j}_{s}\not\in\mathrm{row}(A(S))$,  the row vectors  of  $Y$  corresponding to a basis of $\mathrm{row}(X)$ along with the  row vectors  of  $Y$  corresponding to a basis of $\mathrm{row}(A(S))$  are  linearly independent. This implies that $\mathrm{rank}(T)+\mathrm{rank}(S)\leqslant r-1$.  So,  by Steps 1, 2, 6,  and 8, we have $m(r-2)+3\leqslant \tau+s\leqslant m(r-k-1)+m(k)+1\leqslant  m(r-k)+m(k)$. Applying  Lemma \ref{rank}\,(i), we find that $k=0$  and thus  $S=\text{\sl{K}}_1$. Hence  $t\geqslant m(r-2)+2$ and thus $\mathrm{rank}(T)\geqslant r-1$. Since $(\bmi{j}_{t_{1}}, \bmi{0})\not\in\mathrm{row}(A(T))$, we find that $\mathrm{rank}(K)\geqslant r+1$, a contradiction.

\noindent{\bf{\textsf{Subcase 2.2.}}}   $(\bmi{j}_{t_{1}}, \bmi{0})\not\in\mathrm{row}(A(T))$ and  $\bmi{j}_{s}\in\mathrm{row}(A(S))$.

Clearly, $S$ has no isolated vertex. Since  $\mathrm{rank}(T)+\mathrm{rank}(S)\leqslant r$,  by Steps 1, 2,  6, and 8, we deduce that $m(r-2)+3\leqslant \tau+s\leqslant m(r-k)+m(k)$. Applying  Lemma \ref{rank}\,(i), we find that $k=0$,   which contradicts $\bmi{j}_{s}\in\mathrm{row}(A(S))$.

\noindent{\bf{\textsf{Subcase 2.3.}}}   $(\bmi{j}_{t_{1}}, \bmi{0})\in\mathrm{row}(A(T))$ and  $\bmi{j}_{s}\not\in\mathrm{row}(A(S))$.

Since  $\mathrm{rank}(T)+\mathrm{rank}(S)\leqslant r$,  by Steps 1, 2,  6, and 8, we deduce that $m(r-2)+3\leqslant \tau+s\leqslant m(r-k)+m(k)+1$. Applying  Lemma \ref{rank}\,(ii), we find that $k\in\{0, 2, r-2\}$. If $k=r-2$, then $T=\text{\sl{K}}_2$, which contradicts  $\tau\geqslant p+3$.  If $k=2$, then $S=\text{\sl{K}}_1\cup \text{\sl{K}}_2$.
If  $a$ and $b$ belong to the copies of   $K_1$ and $K_2$ in  $S$, respectively,  then by Corollary  \ref{lova}, $\tau\leqslant|\mathnormal{\Delta}(a,b)|\leqslant p+2$,  which is  a contradiction to   Step 1. Hence $k=0$, that is,  $S=\text{\sl{K}}_1$ and $\tau\geqslant m(r-2)+2$.

\noindent{\bf{\textsf{Subcase 2.4.}}}   $(\bmi{j}_{t_{1}}, \bmi{0})\in\mathrm{row}(A(T))$ and  $\bmi{j}_{s}\in\mathrm{row}(A(S))$.

Obviously, $S$ has no isolated vertex. Choose $\mathrm{rank}(T)-1$ linearly independent row vectors  of  $A(T)$ in such a way that they do not generate  $(\bmi{j}_{t_{1}}, \bmi{0})$. Now, the row vectors  of $A(K)$  corresponding to these  row vectors   together with   the row vectors  of $A(K)$  corresponding to a basis for $\mathrm{row}(A(S))$ are  linearly independent.  This  yields  that $\mathrm{rank}(T)+\mathrm{rank}(S)\leqslant r+1$.  So,   by Steps 1, 2, 6,  and 8,  we have $m(r-2)+3\leqslant \tau+s\leqslant m(r-k+1)+m(k)$. Applying  Lemma \ref{rank}\,(ii), we find that $k\in\{2, 3, r-2\}$.
If $k=r-2$, then $T=\text{\sl{K}}_3$ and we may assume without loss of generality  that $t_1=2$. Then by Lemma \ref{lov0}\,(ii), $\mathrm{rank}(\mathbb{G}-T_1)\leqslant r-1$. However,  this  contradicts the minimality of $\mathbb{G}$  as $\mathbb{G}-T_1$ is a reduced graph of order $m(r)-1$.
Therefore, $k\in\{2, 3\}$,  which means that either   $S=\text{\sl{K}}_2$ or $S=\text{\sl{K}}_3$. Using  Step 1,  if $S=\text{\sl{K}}_2$, then $\tau\geqslant m(r-2)+1$, and if $S=\text{\sl{K}}_3$, then $\tau=m(r-2)$ and $p=m(r-2)-3$, as desired.}
\end{proof}

Now we are in the position to prove  our main theorem.

\begin{thm}\label{mainthm}
Assume  that Conjecture \ref{ackconj} is valid  for all reduced graphs of   rank  at most $46$.
Then Conjecture \ref{ackconj} is true for every  reduced graph.
\end{thm}

\begin{proof}{
Assume that  $r\geqslant47$.
Let $\rho=\rho(\mathbb{G})$ and
$L$ be an induced  subgraph of $\mathbb{G}$ with  $|L|=n-\rho$ and    $\mathrm{rank}(L)<r$.
By Lemma \ref{lov0}\,(iv),  $\mathrm{rank}(L)\geqslant r-2$.
We consider the  following two  cases.

\noindent{\bf{\textsf{Case 1.}}} $\mathrm{rank}(L)=r-2$.

If $H$ has no  duplicated vertices, then by Lemma  \ref{m2} and the minimality of $\mathbb{G}$,
$$\tfrac{m(r)-1}{2}=\tfrac{n}{2}-1<|L|-1\leqslant m(r-2),$$
a contradiction. Hence  $L$ has  duplicated vertices and so $L=H$. Furthermore, by  Lemma \ref{m2} and Theorem \ref{aslasl},  we  obtain  that   $m(r-2)\leqslant\tau=\rho\leqslant m(r-2)+1$.
First suppose  that  $\tau=m(r-2)$. By Theorem \ref{aslasl}, $S=\text{\sl{K}}_3$ and so $p=m(r-2)-3$.
For any pair   $i, j\in\{1, 2, 3\}$, $\mathnormal{\Delta}(\upsilon_i, \upsilon_j)$ contains at least $p-1$ vertices of $P$. It follows that every vertex of $S$ has at most three  neighbors in $P$ and so $p\leqslant 7$ implying that $m(r-2)\leqslant10$,  which is impossible for $r\geqslant8$.

Next suppose  that  $\tau=m(r-2)+1.$ By Theorem \ref{aslasl},  $S=\text{\sl{K}}_2$ and hence    $p=m(r-2)-2$.
Obviously,    $|N_P(\upsilon_1)\cap N_P(\upsilon_2)|\leqslant1$ and thus   for either $\upsilon_1$ or $\upsilon_2$, say $\upsilon_1$, we have $|N_P(\upsilon_1)|\leqslant p/2$.
We may assume that $t_1\leqslant t_2$ implying that
$|N_T(\upsilon_1)|\leqslant(\tau-1)/2$.  Hence  $|N(\upsilon_1)|\leqslant m(r-2)+1$. By Lemma \ref{lov0}\,(i),
$\mathbb{G}- N(\upsilon_1)$ is  of rank at most $r-2$ with an isolated vertex and   no  duplicated vertices. This means that $n\leqslant m(r)$, a contradiction.

\noindent{\bf{\textsf{Case 2.}}} $\mathrm{rank}(L)=r-1$.

By Lemma \ref{lov0}\,(iv), $L$ is necessarily reduced.
From  Lemma \ref{a1}, $\rho<\big(1-\frac{1}{\sqrt{2}}\big)n$ and therefore   $|L|>\tfrac{n}{\sqrt{2}}>5\cdot2^{(r-4)/2}-2$.
Thus Lemma \ref{a1} implies that   $\rho(L)<\big(1-\tfrac{1}{\sqrt{2}}\big)|L|$.
Let $L_0$ be an induced  subgraph of $L$ with $|L_0|=|L|-\rho(L)$  and  $\mathrm{rank}(L_0)<\mathrm{rank}(L)$.
Put $T_0=\mathbb{G}- L_0$ and  $t_0=|T_0|$.
We have $|L_0|>\tfrac{1}{\sqrt{2}}|L|>\tfrac{n}{2}$ and $t_0<\tfrac{n}{2}$. If $L_0$ has no duplicated vertices, then $\tfrac{n}{2}-1<|L_0|-1\leqslant m(r-2)$,  a contradiction. So,  $L_0$ has duplicated vertices which in turn implies that  $\mathrm{rank}(L_0)=r-3$ by Lemma \ref{lov0}\,(iv).
Hence  $\tau\leqslant t_0\leqslant m(r-2)+1$.
Using    Lemma \ref{t1g}  and Theorem \ref{aslasl},   it follows that $\tau\geqslant m(r-2)$.
Therefore,  either $t_0=\tau$ or $t_0=\tau+1$.
Moreover, since $\tau(L)=\rho(L)$ and $\mathrm{rank}(L_0)=\mathrm{rank}(L)-2$, applying Lemma \ref{lov}\,(iii) for $L$, we deduce that  each   duplication class of  $L_0$ consists of two vertices.

We claim that any two  vertices  from two  distinct duplication classes of  $L_0$ are adjacent.
By  contradiction,  suppose that  $U_1=\{u_1, u_1'\}$ and $U_2=\{u_2, u_2'\}$ are two  distinct duplication classes of  $L_0$ with no edges
between them.  Let
$Q= V(T_0)\cap\mathnormal{\Delta}(u_1, u_1')\cap\mathnormal{\Delta}(u_2, u_2')$.
In a similar manner  to the one  used in the proof of Lemma \ref{lov}\,(iii),   we can show that there exist two  disjoint sets $Q_1$ and  $Q_2$  such that $Q=Q_1\cup Q_2$, $Q_1\subseteq N(u_i)\setminus N(u_i')$ and  $Q_2\subseteq N(u_i')\setminus N(u_i)$,  for  $i=1,2$.
From  $t_0\leqslant\tau+1$,  we deduce that for every   duplication class $\{x, y\}$   of $L_0$,  there is at most one  vertex of $T_0$  which is not in $\mathnormal{\Delta}(x,  y)$.
This yields that   $|T_0-Q|\leqslant2$. Furthermore, by the maximality of $L_0$, it is easy to find two  vertices $w_1\in U_1$ and $w_2\in U_2$   such that  at most one   vertex of $T_0-Q$ is contained  in
$\mathnormal{\Delta}(w_1, w_2)$.  Hence
\begin{eqnarray*}
\tau\leqslant|\mathnormal{\Delta}(w_1, w_2)|\leqslant\left\{\begin{array}{lcc}
|L_0|-4 &  \hspace{-5mm}  \text{ if }  t_0=\tau, \\
|L_0|-3 &  \hspace{1.5mm}  \text{ if }  t_0=\tau+1.\\
\end{array}\right.
\end{eqnarray*}
This implies  that $\tau\leqslant n-\tau-4$,
which  contradicts  $\tau\geqslant m(r-2)$. This establishes  the claim.

From the previous paragraph,   it follows that $L_0$ contains  a copy of   $\text{\sl{K}}_\ell$, where $\ell$ is the number of duplication classes  of $L_0$.
Since $\mathrm{rank}(L_0)= r-3$, we conclude that    $\ell\leqslant r-3$.
Thus
$$n-1-\big(m(r-2)+1\big)-(r-3)\leqslant n-1-t_0-\ell\leqslant m(r-3).$$
This  in turn implies that $m(r-2)\leqslant m(r-3)+r-4$, which is  impossible for $r\geqslant10$.

Therefore, we obtain contradictions in both cases and the proof is complete.}
\end{proof}

We finally  mentation that,  similar to  the proofs of Lemmas \ref{phius} and \ref{a1},  one  can verify  the following Lemmas.

\begin{lem}\label{8phius}
For every integer  $n\geqslant2$, $\text{\sl{M}}\big(n, \cos^{-1}(\sqrt{2}-1)\big)<5\cdot2^{(n+2)/2}-2$.
\end{lem}

\begin{lem}\label{8a1}
Let $G$ be a reduced graph of order $n$ and rank $r$. If $n\geqslant5\cdot2^{(r+3)/2}-2$, then $\rho(G)<\big(1-\tfrac{1}{\sqrt{2}}\big)n$.
\end{lem}

For every integer $r\geqslant2$, define  $m'(r)=8m(r)+14$.
Notice that  $m'(r)=2m'(r-2)+2$, whenever
$r\geqslant4$. Now, using this equality,
Lemmas \ref{8phius}  and  \ref{8a1} as well as the approach given in this section, we are able to establish  the following theorem.

\begin{thm}\label{8main}
For every integer $r\geqslant2$, the order of any     reduced  graph of  rank $r$
is at most    $m'(r)$.
\end{thm}

\section*{Acknowledgments}
This research was in part supported by grants from IPM to the first author (No. 92050114) and the second author
(No. 92050405).  The authors are grateful to anonymous referees for valuable comments on an earlier draft of this article.

{}

\end{document}